\documentclass{amsart}

\usepackage{graphicx}

\newtheorem{theorem}{Theorem}

\theoremstyle{definition}

\theoremstyle{remark}
\newtheorem{remark}{Remark}
\newtheorem{example}{Example}

\begin{document}

\title{Constructing Goeritz matrix from Dehn coloring matrix}

\author{Masaki Horiuchi}
\address{Fukushima Prefectural Fukushima High School, 5-72 Moriai-cho Fukushima-city Fukushima 960-8002 Japan}

\author{Kazuhiro Ichihara}
\address{Department of Mathematics, College of Humanities and Sciences, Nihon University, 3-25-40 Sakurajosui, Setagaya-ku, Tokyo 156-8550, Japan}
\email{ichihara.kazuhiro@nihon-u.ac.jp}

\author{Eri Matsudo}
\address{The Institute of Natural Sciences, Nihon University, 3-25-40 Sakurajosui, Setagaya-ku, Tokyo 156-8550, Japan}
\email{matsudo.eri@nihon-u.ac.jp}

\author{Sota Yoshida}
\address{Fukushima Prefectural Fukushima High School, 5-72 Moriai-cho Fukushima-city Fukushima 960-8002 Japan}



\date{\today}

\begin{abstract}
In 1933, L.Goeritz introduced an integral matrix associated to a knot diagram, which is now called a Goeritz matrix, and proved that the absolute value of the determinant of the matrix gives an invariant of a knot. Recently, it was shown that the Goeritz matrix is closely related to the Dehn coloring of knots. In this paper, for a knot diagram, we give an algorithm to construct a Goeritz matrix from a Dehn coloring matrix, from which Dehn colorings are induced, with some geometric information of the diagram. Moreover, if the knot diagram is prime, we give a purely algebraic construction of a Goeritz matrix from a Dehn coloring matrix.
\end{abstract}

\maketitle

\section{Introduction}
In \cite{Goeritz}, Goeritz introduced an integral matrix associated to a knot diagram, which is now called a \textit{Goeritz matrix}, and proved that the absolute value of the determinant of the matrix gives an invariant of a knot. 
It was suggested in \cite{LameySilverWilliams} that the Goeritz matrix is closely related to the \textit{Dehn coloring} of knots in the study of coloring of plane graphs. 
But only alternating knots were considered there. 
In \cite{Traldi}, as an extension and a refinement of the results obtained in \cite{LameySilverWilliams} and also in \cite{Nanyes}, the solution space of the equations with Goeritz matrix (precisely, unreduced Goeritz matrix called in \cite{Traldi}) as a coefficient matrix is isomorphic to the linear space consisting of the Dehn colorings for a knot (\cite[Theorem 4.3]{Traldi}).

In this paper, for a knot diagram, we give an algorithm to construct a Goeritz matrix from a Dehn coloring matrix, from which Dehn colorings are induced, with some geometric information of the diagram (Theorem~\ref{Thm1}). 
Moreover, if the knot diagram is prime, we give a purely algebraic construction of a Goeritz matrix from a Dehn coloring matrix (Theorem~\ref{Thm2}). 

In the rest of this section, we will briefly review some basic definitions. 

Let us start with the definition of the Goeritz matrix for a knot diagram. 
Let $D$ be a knot diagram endowed with a checkerboard coloring on the complementary regions. 
Fo simplicity, we always assume that the unbounded region is shaded by the checkerboard colorings. 
Assign $+1$ or $-1$ to each of the crossings of $D$ as illustrated in Figure~\ref{FigGi} depending upon the position of shaded regions around the crossings. 
We call the $+1$ or $-1$ the \textit{Goeritz index} of the crossing following \cite{Traldi}. 
\begin{figure}[htbt]
  \begin{center}
    {\unitlength=1mm
  \begin{picture}(80,22)
   \put(25,4){\includegraphics[width=3cm]{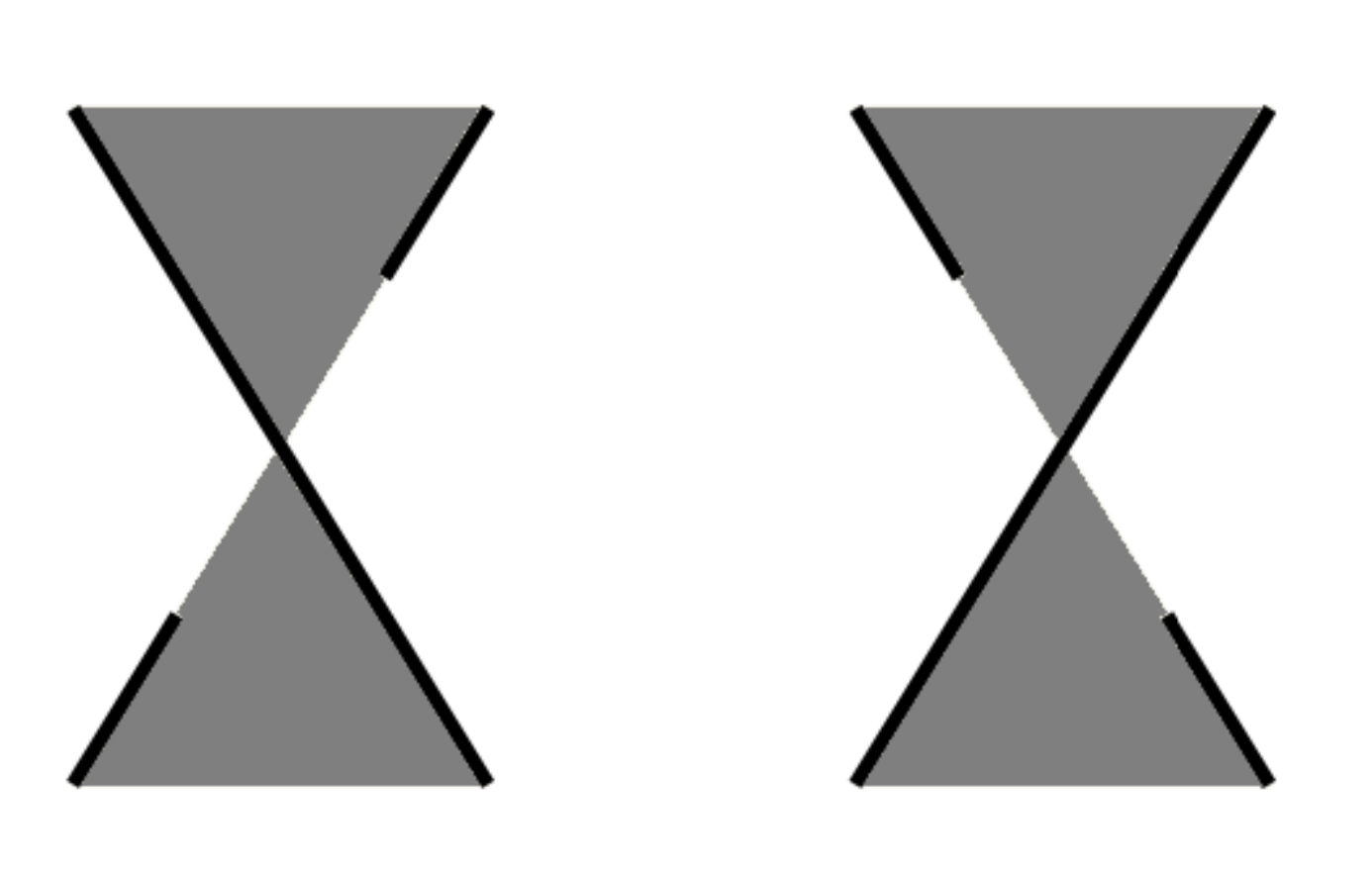}}
   \put(29,-1){$-1$}
   \put(47,-1){$+1$}
  \end{picture}}
\caption{Goeritz index of a crossing}\label{FigGi}
  \end{center}
\end{figure}

For the shaded regions of $D$, we set variables $X_1,\dots,X_n$. 
For a pair of shaded regions corresponding to the variables $X_j$ and $X_k$ ($j \ne k$), 
let $c_{jk}$ denote the sum of the Goeritz indices for the crossings, where the two regions meet. 

Then the matrix $G=(g_{ij})$ defined by the following is called the \textit{Goeritz matrix} of the diagram $D$. 
$$g_{jk} = \left\{
\begin{array}{ll}\displaystyle
c_{jk} & (j \neq k)\\
-\Sigma_{\ell \ne j} c_{j\ell} & (j=k)
\end{array}
\right.$$

\begin{remark}
The matrix $G=(g_{ij})$ obtained above is a $b \times b$ matrix, where $b$ is the number of the shaded regions of $D$. 
In some literature, this matrix is called the pre-Goeritz matrix, and then the matrix obtained from that by removing one row and one column is called the Goeritz matrix. 
The determinant of this Goeritz matrix (latter) is known as an invariant of a knot, which is now called the \textit{determinant} (or the knot determinant) of the knot. 
There are several ways to compute this invariant, one of which is given from the Dehn coloring matrix defined later. 
\end{remark}

\begin{example}\label{Ex2}
Let us consider the knot $8_{19}$ in the well-known knot table. 
The knot has the diagram $D$ illustrated in Figure~\ref{Fig4} together with the checkerboard coloring  with the unbounded region shaded. 
In fact, the diagram $D$ is a minimal diagram of the knot, for the knot is actually the torus knot of type $(3,4)$. 

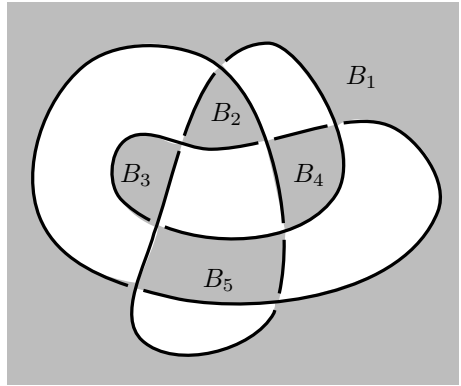
\begin{figure}[htbt]
  {\unitlength=1mm
  \begin{picture}(80,50)
   \put(10,0){\includegraphics[width=6cm]{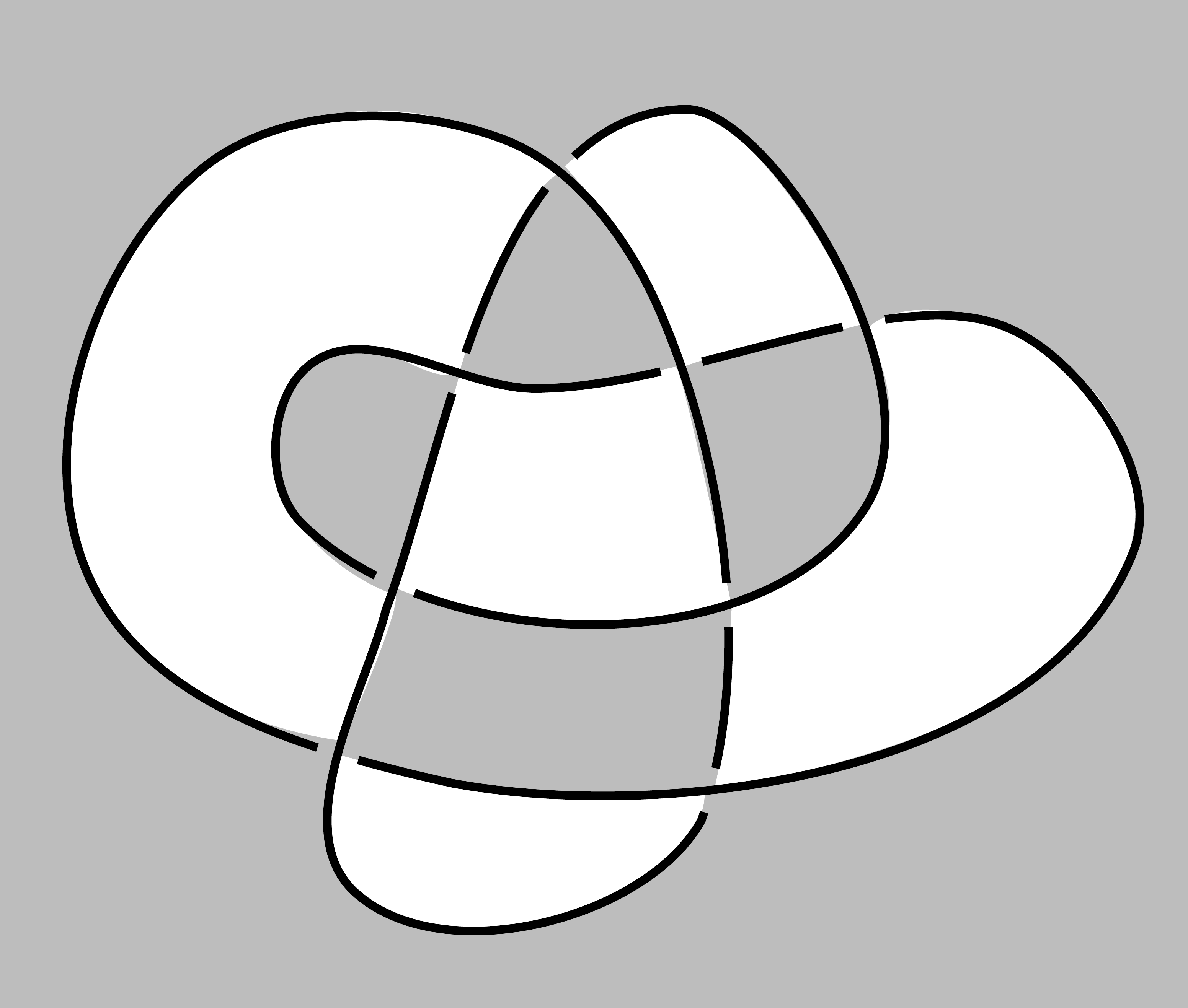}}
   \put(55,40){$B_1$}
   \put(37,35){$B_2$}
   \put(25,27){$B_3$}
   \put(48,27){$B_4$}
   \put(36,13){$B_5$}
  \end{picture}}
\caption{A diagram of $8_{19}$ with a checkerboard coloring.}
\label{Fig4}
\end{figure}

Let $B_1,\dots,B_5$ be the shaded regions and set variables $X_1,\dots,X_5$ corresponding to them  depicted in Figure~\ref{Fig4}. 
Then the Goeritz matrix is obtained as follows. 
\[
\begin{pmatrix}
4 & -1 & 0 & -1 & -2 \\
-1 & -1 & 1 & 1 & 0 \\
0 & 1 & -2 & 0 & 1 \\
-1 & 1 & 0 & -1 & 1 \\
-2 & 0 & 1 & 1 & 0
\end{pmatrix}
\]
\end{example}

\bigskip

Next we give the definition of Dehn coloring equations and a Dehn coloring matrix. 
Let $D$ be an oriented knot diagram of a knot with $m$ complementary regions, say, $R_1, \dots, R_m$. 
We set variables $Z_1, \dots , Z_m$ corresponding to the regions $R_1, \dots, R_m$ of $D$. 
At each crossing of $D$, we consider an equation $Z_i + Z_j - Z_k - Z_l = 0$, where the variables $Z_i, Z_j, Z_k, Z_l$ correspond to the regions $R_i,R_j,R_k,R_l$ respectively adjacent to the crossing of $D$ depicted in Figure~\ref{Fig1}. 
We call the system of such equations the \textit{Dehn coloring equations} for $D$. 
Then the \textit{Dehn coloring matrix} of $D$ is defined as the coefficient matrix of the Dehn coloring equations. 

\begin{figure}[htbt]
  {\unitlength=1mm
  \begin{picture}(65,18)
   \put(30,0){\includegraphics[width=1.7cm]{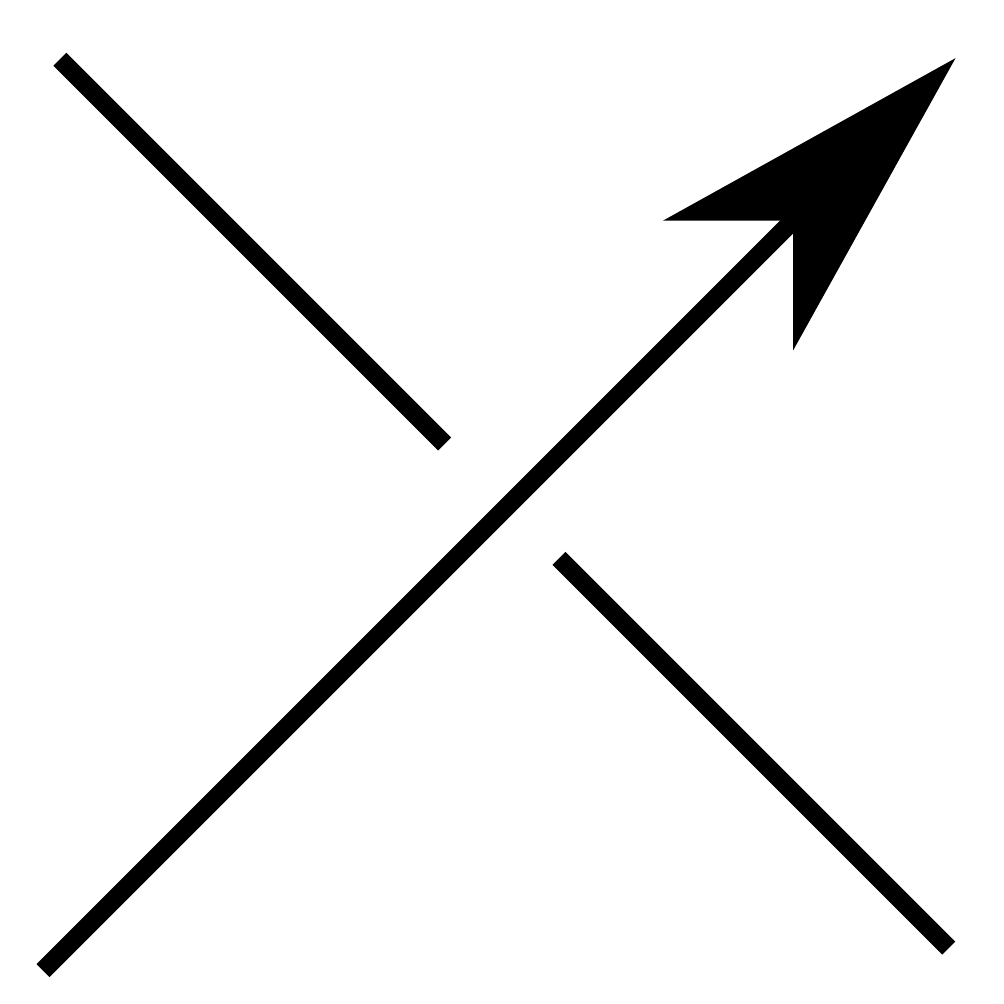}}
   \put(37,15){$R_i$}
   \put(45,7){$R_j$}
   \put(37,0){$R_k$}
   \put(29,7){$R_l$}
  \end{picture}}
\caption{}\label{Fig1}
\end{figure}

\begin{remark}
For a knot diagram $D$, the number of the equations of the Dehn coloring equations is equal to the number of crossings of $D$, say $c(D)$. 
Thus the number of the rows of the Dehn coloring matrix is also equal to $c(D)$. 
Moreover, by the calculating the Euler characteristics, we see that the number of the columns of the Dehn coloring matrix is equal to $c(D)+2$, that is, the number of the rows of the matrix plus two. 
\end{remark}

\begin{example}\label{Ex1}
Consider the diagram $D$ of the knot $8_{19}$ given in Example~\ref{Ex2}. 
Let $R_1, R_2, \dots,R_9,R_0$ be the complementary region of $D$ as in Figure~\ref{Fig2}, and set the corresponding variables $Z_1, Z_2, \dots,Z_9,Z_0$. 

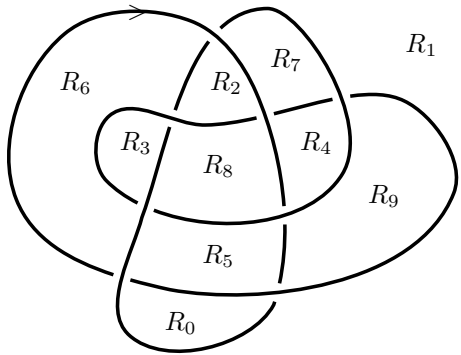
\begin{figure}[htbt]
  {\unitlength=1mm
  \begin{picture}(80,46)
   \put(15,0){\includegraphics[width=6cm]{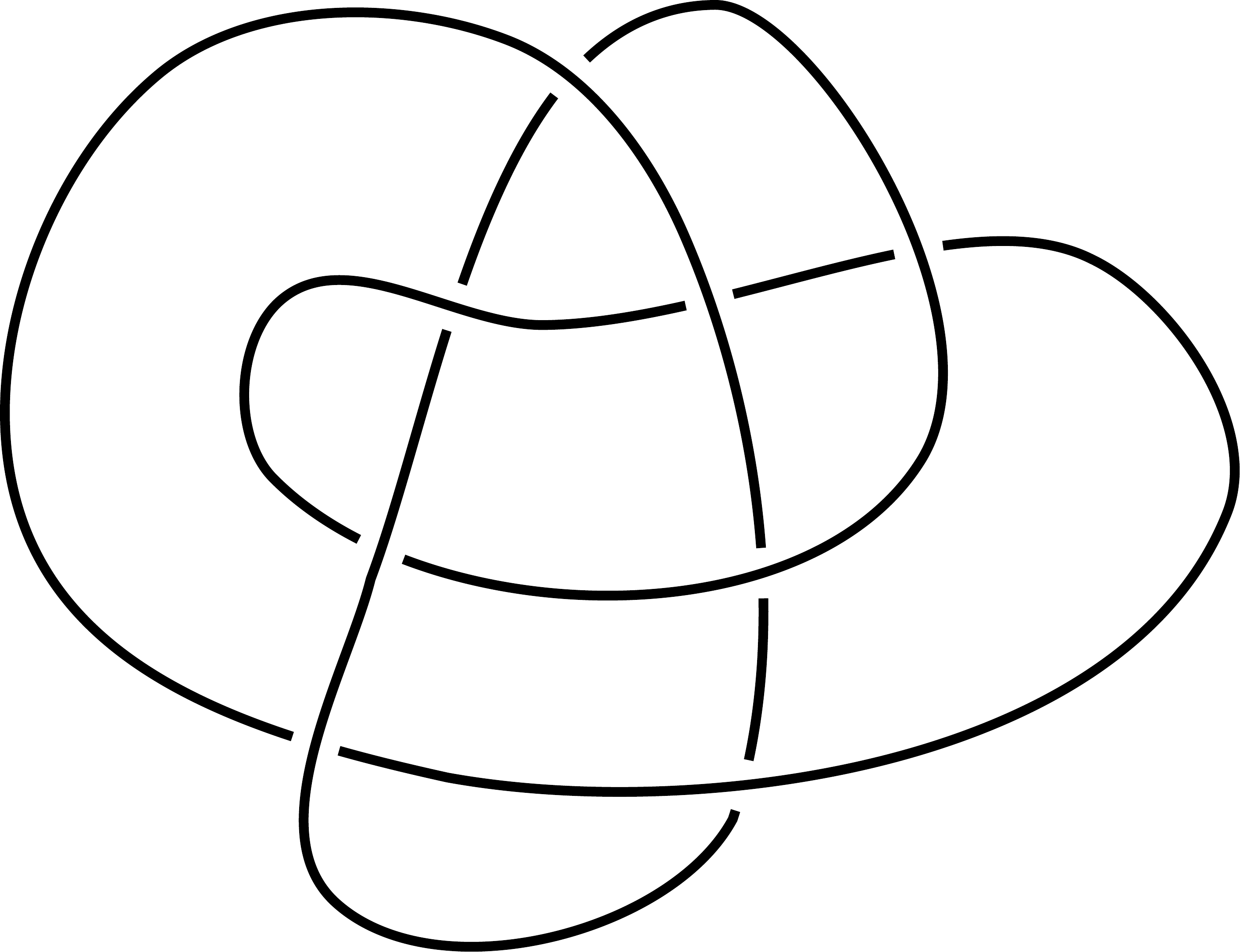}}
   \put(68,40){$R_1$}
   \put(22,35){$R_6$}
   \put(42,35){$R_2$}
   \put(50,38){$R_7$}
   \put(30,27){$R_3$}
   \put(41,24){$R_8$}
   \put(54,27){$R_4$}
   \put(41,12){$R_5$}
   \put(63,20){$R_9$}
   \put(36,3){$R_0$}
   \put(31,44.5){$>$}
  \end{picture}}
\caption{Knot $8_{19}$}
\label{Fig2}
\end{figure}

The Dehn coloring equations for $D$ is obtained as follows. 
\begin{align*}
Z_2 + Z_7 -Z_1 - Z_6 &= 0 \\
Z_4 + Z_8 -Z_2 - Z_7 &= 0 \\
Z_5 + Z_6 -Z_0 - Z_1 &= 0 \\
Z_3 + Z_8 -Z_5 - Z_6 &= 0 \\
Z_4 + Z_9 -Z_1 - Z_7 &= 0 \\
Z_5 + Z_8 -Z_4 - Z_9 &= 0 \\
Z_2 + Z_8 -Z_3 - Z_6 &= 0 \\
Z_0 + Z_5 -Z_1 - Z_9 &= 0
\end{align*}
Here the order of the equations is set as the order of the crossings with the over-arcs passed when one runs on $D$. 

The Dehn coloring matrix for $D$ is as follows. 
\[
\begin{pmatrix}
-1 & 1 & 0 & 0 & 0 & -1 & 1 & 0 & 0 & 0 \\
0 & -1 & 0 & 1 & 0 & 0 & -1 & 1 & 0 & 0 \\
-1 & 0 & 0 & 0 & 1 & 1 & 0 & 0 & 0 & -1 \\
0 & 0 & 1 & 0 & -1 & -1 & 0 & 1 & 0 & 0 \\
-1 & 0 & 0 & 1 & 0 & 0 & -1 & 0 & 1 & 0 \\
0 & 0 & 0 & -1 & 1 & 0 & 0 & 1 & -1 & 0 \\
0 & 1 & -1 & 0 & 0 & -1 & 0 & 1 & 0 & 0 \\
-1 & 0 & 0 & 0 & 1 & 0 & 0 & 0 & -1 & 1 
\end{pmatrix}
\]
\end{example}

\bigskip

\begin{remark}
As a supplement, we give a brief explanation of the Dehn coloring of a knot. 
See \cite{CarterSilverWilliams} for more details. 
Let $p$ be an integer at least 2. 
If the Dehn coloring equation for a knot diagram $D$ has a solution modulo $p$, then the solution induces a map $C: \{ R_1,\dots,R_m \} \to \mathbb{Z} / p \mathbb{Z}$, where $R_1,\dots,R_m$ denote the complementary regions of $D$. 
Such a map $C$ is called a \textit{Dehn $p$-coloring} on $D$. 
Given a Dehn $p$-coloring, $C(R_i)$ is called a \textit{color} of a region $R_i$ by $C$, and if all the regions have the same color by a Dehn $p$-coloring, then such a coloring is called \textit{trivial}. 
On the other hand, for any $p$, there exists a Dehn coloring $C: \{ R_1,\dots,R_m \} \to \{ 0 , 1 \} \subset \mathbb{Z} / p \mathbb{Z}$ corresponding to a checkerboard coloring on the complementary regions $R_1,\dots,R_m $ of the knot diagram. 
If a knot diagram $D$ admits a non-trivial Dehn $p$-coloring not coming from a checkerboard coloring, it can be shown that any other diagram obtained from $D$ by performing Reidemeister moves repeatedly admits such a coloring. 
It follows that admitting a non-trivial Dehn $p$-coloring not coming from a checkerboard coloring is a property of a knot independent from the choice of diagrams. 
Thus one can say that a knot with a diagram admitting such a coloring is \textit{Dehn $p$-colorable}. 
It is known that a knot is Dehn $p$-colorable if and only if it is Fox $p$-colorable, which is well-known and we omit the definition here. 
Moreover it is known to be equivalent that $p$ is a divisor of the determinant of the knot. 
\end{remark}

\section{Results}
The following is the main result of this paper. 

\begin{theorem}\label{Thm1}
Let $M_D$ be the Dehn coloring matrix of a knot diagram $D$ with the checkerboard coloring where the unbounded region is shaded. 
Suppose that the first to the $b$-th columns of $M_D$ correspond to the shaded regions of $D$. 
(that is, $b$ is the number of the shaded regions of $D$.) 
Take the rows of $M_D$ whose $j$-th entries are nonzero. 
Multiply $\pm 1$ to each of the rows so that the $j$-th entry of the obtained row vector coincides with the negative of the Goeritz index of the crossing of the diagram corresponding to the row, and sum them up to obtain a row vector. 
Repeat the procedure for $j=1$ to the number of the columns of $M_D$. 
Then the left half of the matrix consisting of the so obtained row vectors, which corresponds to the first to the $b$-th column, is equal to the Goeritz matrix of $D$. 
Moreover the entries of the rest of the obtained matrix are all zero. 
\end{theorem}

\begin{proof}
We set variables $X_1,\dots,X_n$ corresponding to the shaded regions $B_1, \dots, B_n$ of a knot diagram $D$ with the checkerboard coloring where the unbounded region is shaded, and $Y_{n+1},\dots,Y_m$ corresponding to the unshaded regions $W_{n+1},\dots,W_m$ of $D$. 
Here $n$ denotes the number of such shaded regions and $m$ the number of the complementary regions of $D$. 
By using the variables $X_1,\dots,X_n,Y_{n+1},\dots,Y_m$, we consider the Dehn coloring equations and let $M_D$ be the Dehn coloring matrix as its coefficient matrix. 

Then, for each equation in the Dehn coloring equations, the two variables corresponding to the shaded regions touching at the crossing associated to the equation must have the coefficients with opposite signs. 

For example, in the case of Figure~\ref{Fig5}, corresponding equations are the following.
\[
X_i + Y_j - X_k - Y_l = 0 , \quad  Y_i + X_j - Y_k - X_l =0
\]
\begin{figure}[htbt]
  {\unitlength=1mm
  \begin{picture}(80,20)
   \put(15,2){\includegraphics[width=5cm]{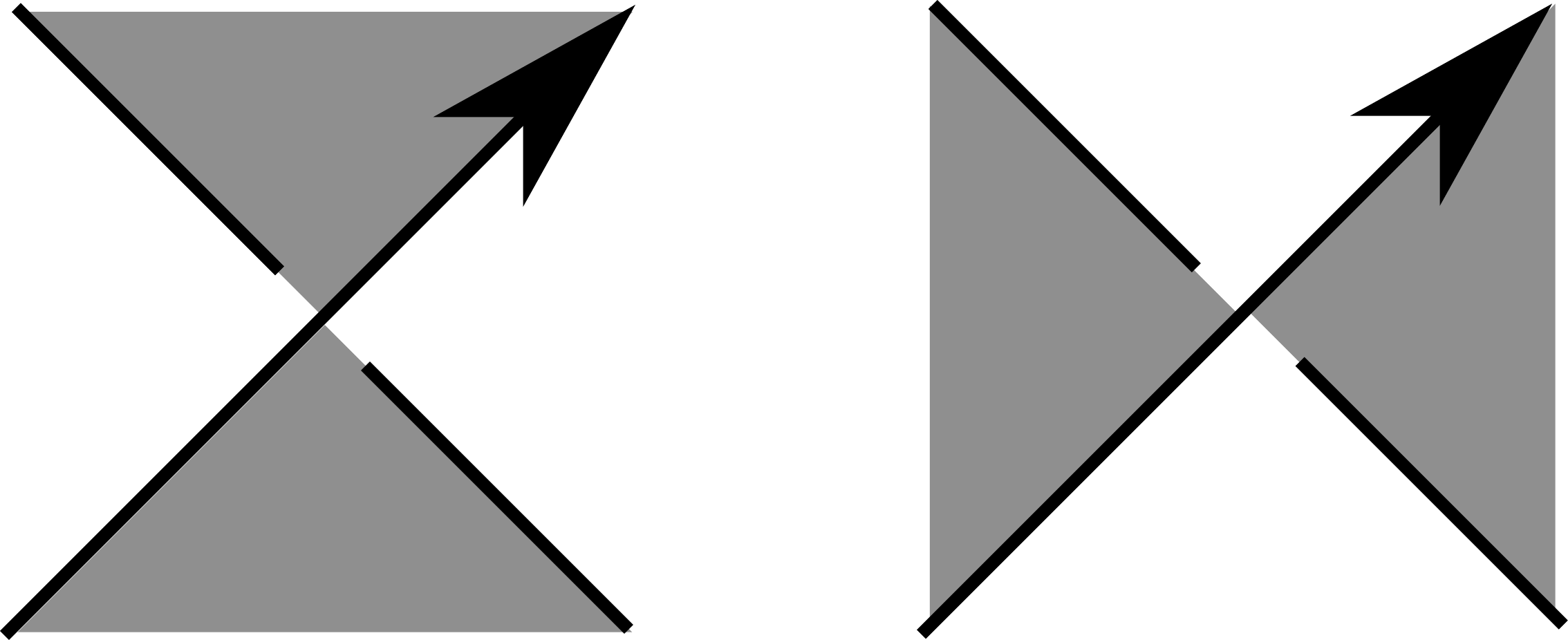}}
   \put(24,18){$B_i$}
   \put(31,11){$W_j$}
   \put(24,4){$B_k$}
   \put(15,11){$W_l$}
   \put(53,18){$W_i$}
   \put(60,11){$B_j$}
   \put(53,4){$W_k$}
   \put(46,11){$B_l$}
  \end{picture}}
\caption{}\label{Fig5}
\end{figure}

In both cases, the signs of the coefficients of the variables corresponding the shaded regions are opposite. 

Now, for the variable $X_j$, suppose that the corresponding region $B_j$ touches the $i$-th shaded region $B_i$ at some crossings. 
Then the $i$-th entry of the row vector obtained in the way stated in the theorem for $X_j$ is coincident to the entry $g_{ji}$ of the Goeritz matrix of $D$, for the $g_{ji}$ is the sum of the Goeritz indices of such crossings.  
Moreover the $j$-th entry of the row vector obtained in that way for $X_j$ is also coincident to the entry $g_{jj}$ of the Goeritz matrix of $D$.
Therefore the left half of the matrix consisting of the so obtained row vectors, which corresponds to the first to the $b$-th column, is equal to the Goeritz matrix of $D$. 

Next, consider the coefficients of the variables $Y_{n+1}, \dots, Y_m$ corresponding to the unshaded regions $W_{n+1}, \dots, W_m$ in the Dehn coloring equations. 
In the following, we show that, after the summing the Dehn coloring equations corresponding to the summing of the row vectors in the way of Theorem~\ref{Thm1}, the coefficients of $Y_k$'s are all zero in the obtained equations. 

Fix a variable $X_j$ corresponding to the region $B_j$, and consider a variable $Y_k$ corresponding to the region $W_k$ which is adjacent to the region $B_j$ along a part of an arc of the knot diagram. 

First, we consider the case that neither $B_j$ nor $W_k$ are monogons. 
Then 
the two equations associated to the endpoints of the part of the arc between $B_j$ and $W_k$  are described as one of the following, up to mirror image, in which the variables correspond to the regions indicated in Figure ~\ref{Fig6} respectively. 

{\small 
\[
\begin{cases}
X_j + Y_k - X' - Y' =0\\
- X_j - Y_k + X'' + Y'' =0
\end{cases}
\begin{cases}
X_j + Y_k - X' - Y' =0\\
X_j - Y_k - X'' + Y'' =0
\end{cases}
\begin{cases}
X_j + Y_k - X' - Y' =0\\
- X_j + Y_k + X'' - Y'' =0
\end{cases}
\]
\[
\begin{cases}
X_j - Y_k - X' + Y' =0\\
X_j - Y_k - X'' + Y'' =0
\end{cases}
\begin{cases}
X_j - Y_k - X' + Y' =0\\
-X_j + Y_k + X'' - Y'' =0
\end{cases}
\]
\[
\begin{cases}
X_j - Y_k - X' + Y' =0\\
- X_j - Y_k + X'' + Y'' =0
\end{cases}
\begin{cases}
-X_j + Y_k + X' - Y' =0\\
-X_j - Y_k + X'' + Y'' =0
\end{cases}
\]
}

\begin{figure}[htbt]
\centering
  {\unitlength=1mm
  \begin{picture}(120,60)
   \put(2,41){\includegraphics[width=7.5cm]{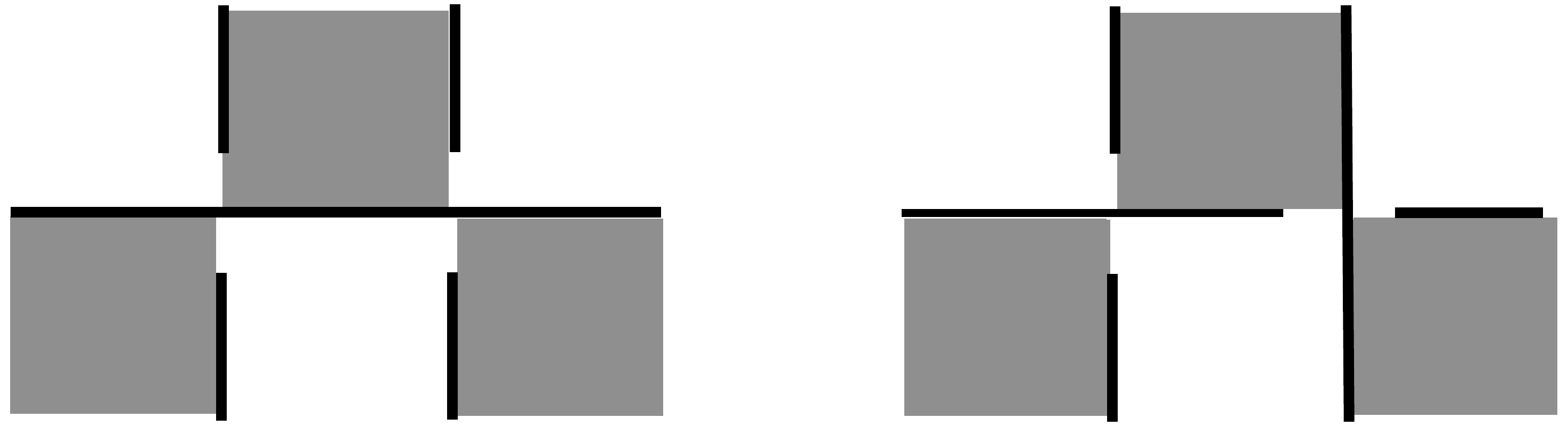}}
   \put(83,41.2){\includegraphics[width=3.6cm]{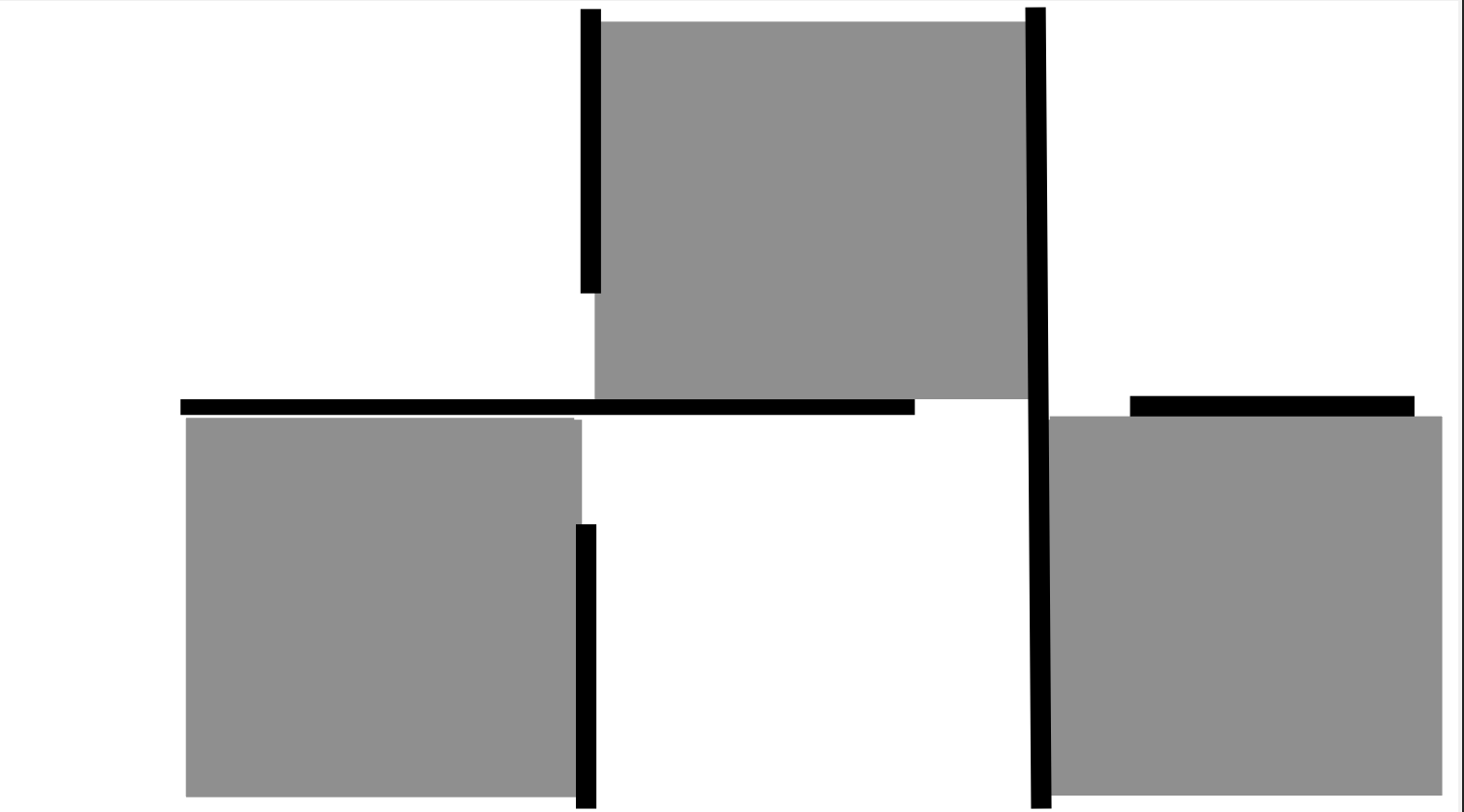}}
   \multiput(2,18)(43,0){2}{\includegraphics[width=3.2cm]{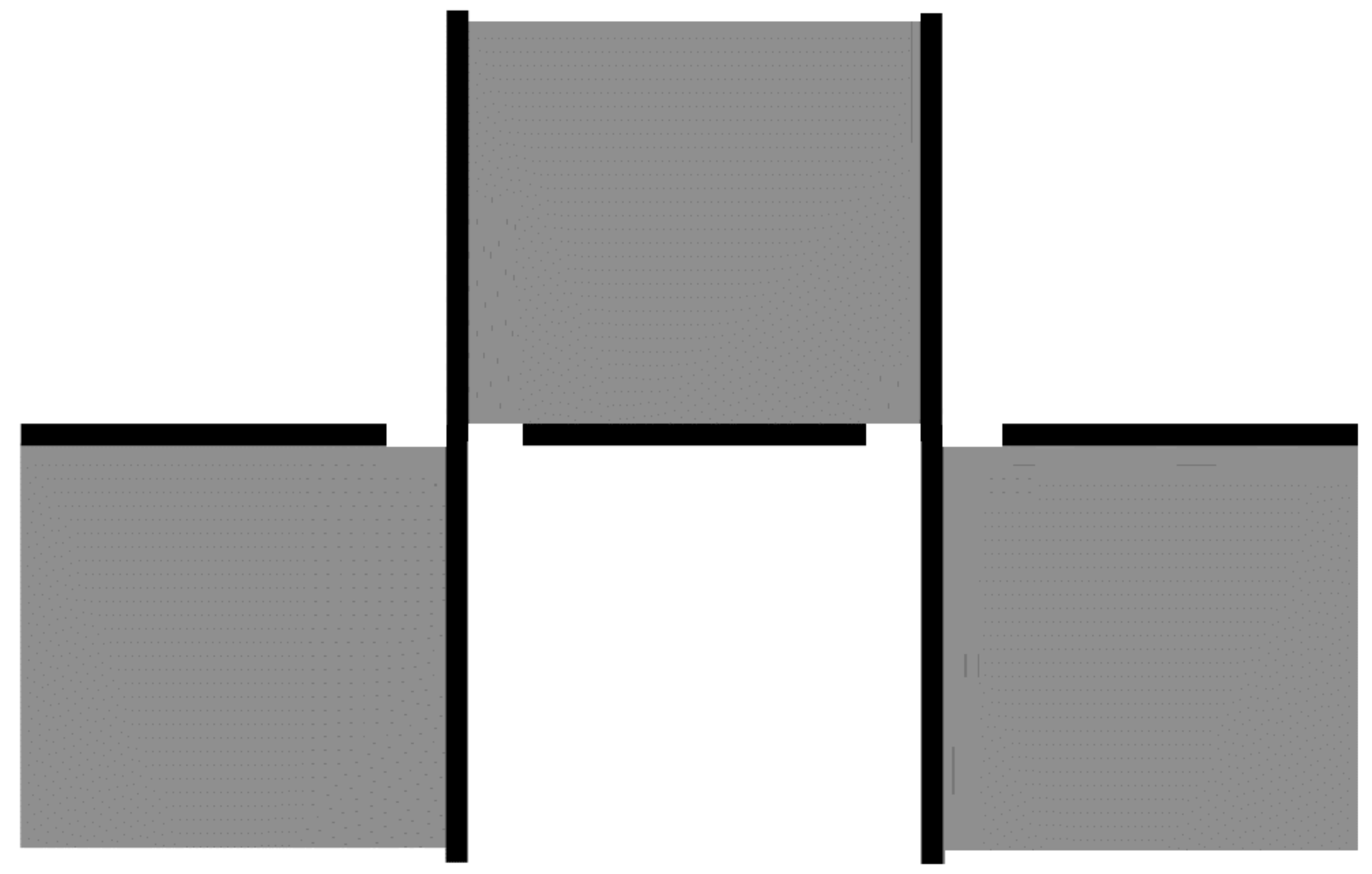}}
   \multiput(2,-4)(43,0){2}{\includegraphics[width=3.2cm]{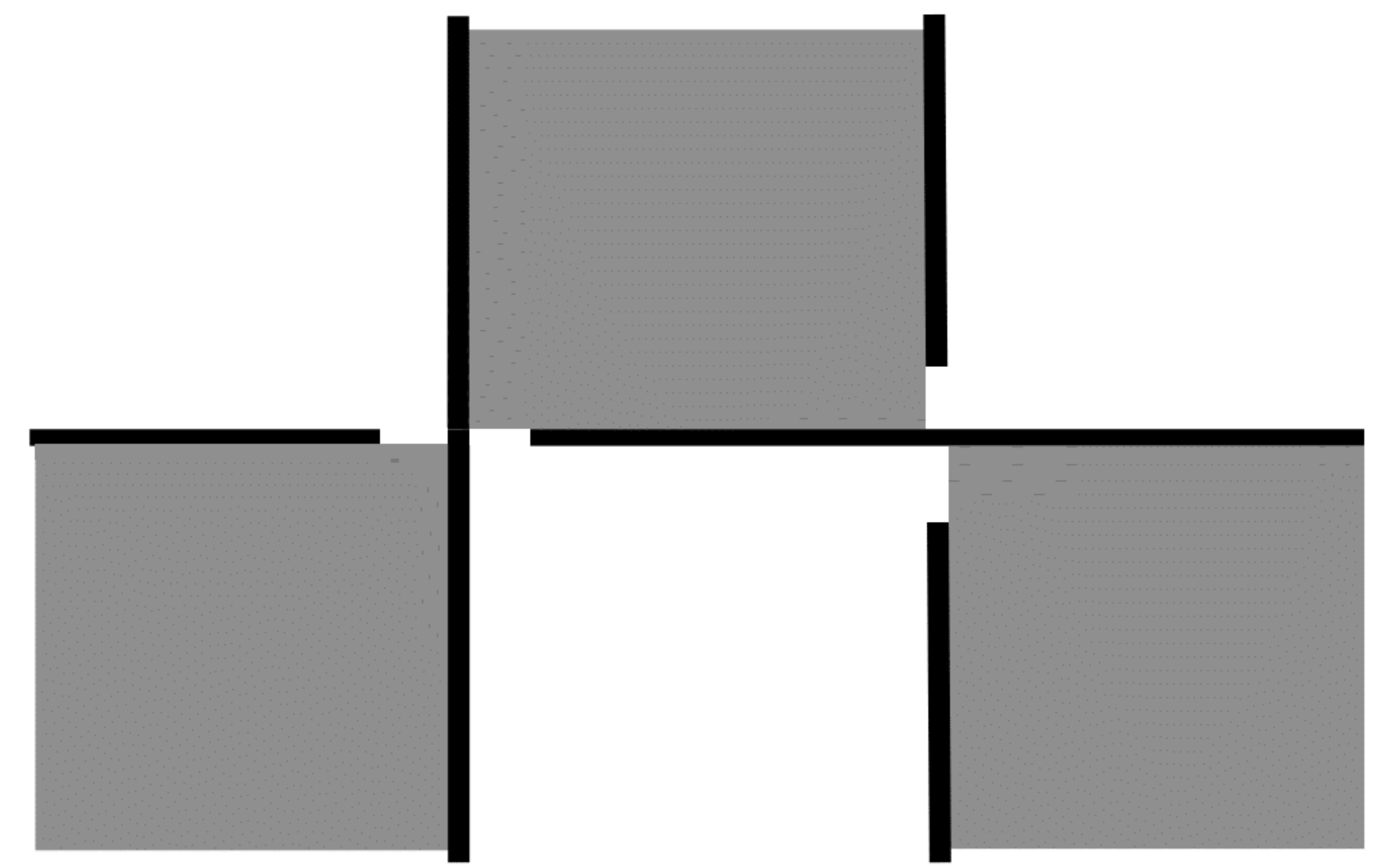}}
   \multiput(16,55)(0,-23){3}{$B_j$}
   \multiput(16,45)(0,-23){3}{$W_k$}
   \multiput(5,45)(0,-23){3}{$B'$}
   \multiput(27,45)(0,-23){3}{$B''$}
   \multiput(5,55)(0,-23){3}{$W'$}
   \multiput(27,55)(0,-23){3}{$W''$}
   \multiput(59,55)(0,-23){3}{$B_j$}
   \multiput(59,45)(0,-23){3}{$W_k$}
   \multiput(48,45)(0,-23){3}{$B'$}
   \multiput(70,45)(0,-23){3}{$B''$}
   \multiput(48,55)(0,-23){3}{$W'$}
   \multiput(70,55)(0,-23){3}{$W''$}
   \put(101,55){$B_j$}
   \put(101,45){$W_k$}
   \put(90,45){$B'$}
   \put(112,45){$B''$}
   \put(90,55){$W'$}
   \put(112,55){$W''$}
   \multiput(19,50.3)(41,0){3}{$>$}
   \multiput(18,5)(43,0){2}{$>$}
   \put(65.3,57){$\wedge$}
   \put(107.5,43){$\vee$}
   \multiput(11.5,33)(0,-23){2}{$\wedge$} 
   \put(22.6,33){$\wedge$} 
   \put(54.5,33){$\wedge$}
   \put(65.55,23){$\vee$}
   \put(54.55,-1){$\vee$}
   \end{picture}}
\caption{}
\label{Fig6}
\end{figure}

In all the cases, after taking the sum of the equations for $X_j$ in the way of Theorem~\ref{Thm1}, the coefficients of $Y_k$ become 0. 
For example, in the case of the left of the top of the figure, since the Goeritz indices of the crossings where the regions $B_j$ and $W_k$ touch are $+1$ (left) and $-1$ (right), the equations are summed as follows in the way of Theorem~\ref{Thm1}. 
\begin{align*}
&- ( X_j + Y_k - X' - Y' ) - ( -X_j - Y_k + X'' + Y'' ) \\
&= X' - X'' + Y' - Y''
\end{align*}
In the same way, when the regions $B_j$ and $W_k$ face along more than two arcs of the diagram $D$, we see that the coefficient of $Y_k$ becomes 0. 

Next,  we consider that the region $B_j$ is a monogon. 
Then the equation associated to the crossing touching $B_j$ is described as the following. 
\[
X_k + Y_k - X_j - Y_k = X_k - X_j = 0
\]
In the case of the region $W_k$ is a monogon, the equation associated to the crossing touching $B_j$ is described as the following. 
\[
X_k + Y_k - X_k - Y_l = Y_k - Y_l = 0
\]
In both cases, each equation can be ignored when the Dehn coloring equations are summed up corresponding to summing the row vectors in the way of Theorem~\ref{Thm1}. 

Thus, after summing the Dehn coloring equations corresponding to summing the row vectors in the way of Theorem~\ref{Thm1}, the coefficients of $Y_k$'s are shown to be all zero in the obtained equations. 
It concludes that all the column vectors from $(b+1)$-th to the last are zero vectors in the matrix obtained in the way of Theorem~\ref{Thm1}. 
\end{proof}

\begin{example}\label{Ex3}
By Example~\ref{Ex1}, for the diagram $D$ of the knot $8_{19}$, the Dehn coloring equations and the Dehn coloring matrix $M_D$ are given as follows. 
Note that the variables $Z_1, \dots, Z_5$ in Example~\ref{Ex1} are changed into $X_1, \dots, X_5$ and $Z_6, \dots, Z_9, Z_0$ are into $Y_1, \dots, Y_4, Y_5$.
\begin{align*}
X_2 + Y_2 -X_1 - Y_1 &= 0 \\
X_4 + Y_3 -X_2 - Y_2 &= 0 \\
X_5 + Y_1 -X_1 - Y_5 &= 0 \\
X_3 + Y_3 -X_5 - Y_1 &= 0 \\
X_4 + Y_4 -X_1 - Y_2 &= 0 \\
X_5 + Y_3 -X_4 - Y_4 &= 0 \\
X_2 + Y_3 -X_3 - Y_1 &= 0 \\
X_5 + Y_5 -X_1 - Y_4 &= 0
\end{align*}
\[M_D=
\begin{pmatrix}
-1 & 1 & 0 & 0 & 0 & -1 & 1 & 0 & 0 & 0 \\
0 & -1 & 0 & 1 & 0 & 0 & -1 & 1 & 0 & 0 \\
-1 & 0 & 0 & 0 & 1 & 1 & 0 & 0 & 0 & -1 \\
0 & 0 & 1 & 0 & -1 & -1 & 0 & 1 & 0 & 0 \\
-1 & 0 & 0 & 1 & 0 & 0 & -1 & 0 & 1 & 0 \\
0 & 0 & 0 & -1 & 1 & 0 & 0 & 1 & -1 & 0 \\
0 & 1 & -1 & 0 & 0 & -1 & 0 & 1 & 0 & 0 \\
-1 & 0 & 0 & 0 & 1 & 0 & 0 & 0 & -1 & 1 
\end{pmatrix}\]

The rows of $M_D$ with non-zero first entries are rows of number 1, 3, 5, 8, and the corresponding equations are as follows. 
\begin{align*}
X_2 + Y_2 -X_1 - Y_1 &= 0 \\
X_5 + Y_1 -X_1 - Y_5 &= 0 \\
X_4 + Y_4 -X_1 - Y_2 &= 0 \\
X_5 + Y_5 -X_1 - Y_4 &= 0
\end{align*}
From Figure~\ref{Fig4}, we see that the Goeritz indices of the crossings corresponding to the rows are
\[ -1, \quad -1, \quad -1, \quad -1 \]
and so, we have the following by taking the sum of them in the way of Theorem~\ref{Thm1}. 
{\small \begin{align*}
&(-1) ( X_2 + Y_2 -X_1 - Y_1 ) + (-1) (X_5 + Y_1 -X_1 - Y_5 ) \\
&+ (-1) (X_4 + Y_4 -X_1 - Y_2 ) + (-1) ( X_5 + Y_5 -X_1 - Y_4 ) \\
&= 4 X_1 -X_2 - X_4 -2 X_5 = 0
\end{align*}}
Repeatedly perform the same procedure from the second to the fifth column, we obtain the following equations. 
\begin{align*}
- X_1 - X_2 + X_3 + X_4 &= 0 \\
X_2 -2 X_3 + X_5 &= 0 \\
X_1 + X_2 - 3 X_4 + X_5 &= 0 \\
-2 X_1 + X_3 + X_4 &= 0
\end{align*}
The coefficient matrix of this equations is as follows. 
\[
\begin{pmatrix}
4 & -1 & 0 & -1 & -2 \\
-1 & -1 & 1 & 1 & 0 \\
0 & 1 & -2 & 0 & 1 \\
-1 & 1 & 0 & -1 & 1 \\
-2 & 0 & 1 & 1 & 0
\end{pmatrix}
\]
This matrix is coincident to the Goeritz matrix of the diagram $D$ of the knot $8_{19}$ obtained in Example~\ref{Ex2}. 
\end{example}

\bigskip

Next we show that, when a given knot diagram is prime, the Goeritz matrix is constructed from the Dehn coloring matrix, up to sign, purely algebraically. 
Here a knot diagram is called \textit{prime} if any simple closed curve transversely intersecting the diagram with only two points bounds a disk which intersects the diagram with single unknotted arc. 

\begin{theorem}\label{Thm2}
Given a knot diagram $D$ with the checkerboard coloring where the unbounded regions is shaded, let $M_D$ be the corresponding Dehn coloring matrix. 
Suppose that $D$ is prime and the columns of $M_D$ corresponding to the shaded regions are the columns from the first to the $b$-th. 
For the $j$-th column of $M_D$, take the rows of $M_D$ whose $j$-th entries are non-zero. 
Then, by multiplying the rows $\pm 1$ and summing up suitably, one can obtain a row vector whose entries from the $(b+1)$-th to the last are all zero, and such a way to obtain the vector is unique up to sign. 
Moreover, by fixing one of the so obtained row vectors, and multiplying the other row vectors by $\pm 1$, one can obtain the Goeritz matrix of $D$ up to sign as a non-zero part of the obtained matrix. 
\end{theorem}

\begin{proof}
In the same way of the proof of Theorem~\ref{Thm1}, we set the variables $X_1,\dots,X_n$ corresponding to the shaded regions $B_1, \dots, B_n$ of the knot diagram $D$ and the variables $Y_{n+1}, \dots , Y_m$ corresponding to the unshaded regions $W_{n+1},\dots,W_m$ of $D$. 
Consider the Dehn coloring equations by using $X_1,\dots,X_n,Y_{n+1},\dots,Y_m$, and let $M_D$ be the Dehn coloring matrix for $D$ as the coefficient matrix of the equations. 

By Theorem~\ref{Thm1}, from the rows of $M_D$ with non-zero $j$-th entries, one can multiply them $\pm 1$ suitably and sum them up so that the entries of the obtained row vectors from $(b+1)$-th to the last are all zero. 
In terms of the Dehn coloring equations, one can suitably take a sum of the Dehn coloring equations so that the coefficient of $W_k$'s of the equations so obtained are all zero. 
In the following, we show that there is a unique way, up to sign, to do so. 

Now, we fix a shaded region $B_j$ corresponding to the variable $X_j$, and consider an unshaded region $W_k$ adjacent to $B_j$ corresponding to the variable $Y_k$. 
Then, by the assumption that the diagram $D$ is prime, the regions $B_j$ and $W_k$ are adjacent along a single arc of $D$. 
This implies that there are only two equations which include the variables $X_j$ and $Y_k$ among the Dehn coloring equations. 
Thus, if one of the equations including $X_j$ is fixed, there is a unique way to multiply $\pm 1$ and sum up the other equations including $X_j$ to obtain an equation whose coefficients of $Y_k$'s are all zero. 
Here, for each $j$, there still remains an ambiguity to fix the sign of the coefficient of $X_j$ in the first equation. 
That is, we do not fix a system of equations in such a way, and so, do not fix a coefficient matrix. 

To fix a matrix, we multiply $\pm 1$ to each obtained equation so that the coefficient matrix of the obtained system of equations with variables $X_1 , \dots , X_b$ to be a symmetric matrix. 
Then we can fix the coefficient matrix up to sign. 
The matrix so obtained must contain the Goeritz matrix or the negative in the left half, since the Goeritz matrix is symmetric by definition. 
\end{proof}

\begin{example}\label{Ex4}
In the same way as Example~\ref{Ex3}, we consider the Dehn coloring matrix $M_D$ of the diagram $D$ of the knot $8_{19}$. 
\[M_D=
\begin{pmatrix}
-1 & 1 & 0 & 0 & 0 & -1 & 1 & 0 & 0 & 0 \\
0 & -1 & 0 & 1 & 0 & 0 & -1 & 1 & 0 & 0 \\
-1 & 0 & 0 & 0 & 1 & 1 & 0 & 0 & 0 & -1 \\
0 & 0 & 1 & 0 & -1 & -1 & 0 & 1 & 0 & 0 \\
-1 & 0 & 0 & 1 & 0 & 0 & -1 & 0 & 1 & 0 \\
0 & 0 & 0 & -1 & 1 & 0 & 0 & 1 & -1 & 0 \\
0 & 1 & -1 & 0 & 0 & -1 & 0 & 1 & 0 & 0 \\
-1 & 0 & 0 & 0 & 1 & 0 & 0 & 0 & -1 & 1 
\end{pmatrix}\]
The rows of $M_D$ with non-zero first entries are rows of number 1, 3, 5, 8. 
Fix the first row of $M_D$, and suitably multiply $\pm 1$ to the rows 3, 7, 8 and sum them up so that the entries from 6-th to 10-th are all 0. 

For example, focusing the $(1,6)$- and $(3,6)$-entries in $M_D$, we should multiply $+1$ to the 3rd row and add it to the first row. 
Similarly, considering the rows 7 and 8, we have the following new row vector. 
\begin{align*}
&\begin{pmatrix}
-1 & 1 & 0 & 0 & 0 & -1 & 1 & 0 & 0 & 0 
\end{pmatrix}
+
\begin{pmatrix}
-1 & 0 & 0 & 0 & 1 & 1 & 0 & 0 & 0 & -1 
\end{pmatrix}\\
&+
\begin{pmatrix}
-1 & 0 & 0 & 1 & 0 & 0 & -1 & 0 & 1 & 0 
\end{pmatrix}
+
\begin{pmatrix}
-1 & 0 & 0 & 0 & 1 & 0 & 0 & 0 & -1 & 1 
\end{pmatrix} \\
&=
\begin{pmatrix}
-4 & 1 & 0 & 1 & 2 & 0 & 0 & 0 & 0 & 0 
\end{pmatrix}
\end{align*}

Repeatedly, performing the same procedure, we obtain the following. 
\begin{align*}
\begin{pmatrix}
-1 & -1 & 1 & 1 & 0 & 0 & 0 & 0 & 0 & 0 
\end{pmatrix}\\
\begin{pmatrix}
0 & -1 & 2 & 0 & -1 & 0 & 0 & 0 & 0 & 0 
\end{pmatrix}\\
\begin{pmatrix}
1 & -1 & 0 & 1 & -1 & 0 & 0 & 0 & 0 & 0 
\end{pmatrix}\\
\begin{pmatrix}
-2 & 0 & 1 & 1 & 0 & 0 & 0 & 0 & 0 & 0 
\end{pmatrix}
\end{align*}

Arranging the row vectors so obtained to get the following matrix. 
\[\begin{pmatrix}
-4 & 1 & 0 & 1 & 2 & 0 & 0 & 0 & 0 & 0 \\
-1 & -1 & 1 & 1 & 0 & 0 & 0 & 0 & 0 & 0 \\
0 & -1 & 2 & 0 & -1 & 0 & 0 & 0 & 0 & 0 \\
1 & -1 & 0 & 1 & -1 & 0 & 0 & 0 & 0 & 0 \\
-2 & 0 & 1 & 1 & 0 & 0 & 0 & 0 & 0 & 0 
\end{pmatrix}\]
However, this is not symmetric, and so, retaking row vectors suitably so that the resultant matrix contains a symmetric matrix. 
Then the following matrix is constructed. 
\[\begin{pmatrix}
-4 & 1 & 0 & 1 & 2 & 0 & 0 & 0 & 0 & 0 \\
1 & 1 & -1 & -1 & 0 & 0 & 0 & 0 & 0 & 0 \\
0 & -1 & 2 & 0 & -1 & 0 & 0 & 0 & 0 & 0 \\
1 & -1 & 0 & 1 & -1 & 0 & 0 & 0 & 0 & 0 \\
2 & 0 & -1 & -1 & 0 & 0 & 0 & 0 & 0 & 0 
\end{pmatrix}\]
The left part of the matrix, that is, the square matrix consisting of the first to the fifth columns, is the negative of the Goeritz matrix obtained in Example~\ref{Ex2}. 
\end{example}

\bigskip

\section*{Acknowledgements}
This article is a result of an SSH (Super Science High-school) project conducted at the Fukushima Prefectural Fukushima High School supported by Japan Science and Technology Agency (JST). 
The authors thank to Norio Hangai, a teacher of Fukushima High School, for his all supports and advice during this program.

\end{document}